\documentclass[11pt]{article}

\usepackage[margin=1in]{geometry}
\usepackage{amsmath,amssymb,amsthm,mathtools,mathrsfs,bm}
\usepackage[T1]{fontenc}
\usepackage[utf8]{inputenc}
\usepackage{lmodern}
\usepackage{microtype}
\usepackage{booktabs}
\usepackage{array}
\usepackage{enumitem}
\usepackage{multirow}
\usepackage{hyperref}
\hypersetup{
  colorlinks=true,
  linkcolor=blue,
  citecolor=blue,
  urlcolor=blue
}

\newtheorem{theorem}{Theorem}[section]
\newtheorem{lemma}[theorem]{Lemma}
\newtheorem{proposition}[theorem]{Proposition}
\newtheorem{corollary}[theorem]{Corollary}
\theoremstyle{definition}
\newtheorem{definition}[theorem]{Definition}
\newtheorem{remark}[theorem]{Remark}

\newcommand{\R}{\mathbb R}
\newcommand{\Z}{\mathbb Z}

\newcommand{\Oh}{O_h}
\newcommand{\norm}[1]{\left\lVert #1 \right\rVert}
\newcommand{\abs}[1]{\left\lvert #1 \right\rvert}

\newcommand{\Tthree}{\mathbb T^3}

\title{%
Orbit-Level Transfer Matrix for the 3D Fourier--Galerkin Navier--Stokes System on the Periodic Torus:\
Explicit Orbit--Triad Incidence Bounds and Deterministic Row-Sum Estimates}

\author{Oleg Kiriukhin\\
City University of Hong Kong\\
\texttt{okiriukh@cityu.edu.hk}}

\date{March 2026}

\begin{document}
\maketitle

\begin{abstract}
I study the cubic Fourier--Galerkin truncation of the three-dimensional (3D) incompressible Navier--Stokes equations on the periodic torus after reduction by the full octahedral symmetry group $O_h$. The nonlinear interaction is encoded by a state-dependent orbit-level transfer matrix $M_N(u)$, and the main discrete problem is to estimate orbit--triad incidences in shell slices of translated cubes. Using a face-normalized decomposition, I reduce the local counting problem to the classical two-squares representation function and obtain an incidence bound of order $N^{4+\varepsilon}$ by the shell-counting argument developed in this manuscript. I also derive the exact orbit-level enstrophy identity, the algebraic decomposition $M_N(u)=A_N(u)+V_N(u)$, and deterministic Sobolev row-sum bounds for the raw matrix $M_N(u)$ in the stated range of exponents. These results give an orbit-level description of nonlinear transfer in the truncated system.
\end{abstract}

\medskip
\noindent\textbf{Keywords:} 3D Navier--Stokes equations, Fourier--Galerkin truncation, periodic torus, octahedral symmetry group, orbit-level transfer matrix, orbit--triad incidence bounds, enstrophy identity, deterministic row-sum bounds.
\section{Introduction}

The three-dimensional incompressible Navier--Stokes equations on the torus admit a natural Fourier representation in which truncation, symmetry reduction, and triadic interaction structure can be analyzed simultaneously. This paper works with the cubic truncation
\begin{equation*}
\Lambda_N:=\{k\in\Z^3\setminus\{0\}:\abs{k}_\infty\le N\},
\end{equation*}
and then passes to the quotient by the full octahedral group $\Oh$ of signed coordinate permutations.

This symmetry reduction organizes the truncated dynamics by $\Oh$-orbits and leads to a state-dependent orbit-level transfer matrix $M_N(u)$. Its algebraic decomposition
\begin{equation*}
M_N(u)=A_N(u)+V_N(u),
\qquad
A_N(u)^T=-A_N(u),
\qquad
V_N(u)^T=V_N(u)
\end{equation*}
separates antisymmetric redistribution from the symmetric part relevant to the associated quadratic form on orbit coordinates. Throughout the paper, all claims are kept at the level justified by explicit counting arguments and by the deterministic Sobolev estimates proved here, in particular, the row-sum bound in Section~\ref{sec:sobolev} is established only for $3/2<s<3$, which is the range supported by the discrete-to-continuum comparison used there. For a related orbit-level stretching analysis emphasizing sharp incidence, spectral decay, and a continuation criterion, see \cite{Kiriukhin2026OrbitLevelStretching}.

The main discrete object is the orbit--triad incidence count. For target and source orbits $\alpha,\beta\in\mathcal O_N$, let $\Gamma_{\alpha\beta}$ denote the number of ordered triads
\begin{equation*}
(k,p,q),
\qquad
k\in\alpha,
\quad p\in\beta,
\quad q=k-p\in\Lambda_N.
\end{equation*}
For fixed $k$, the admissible set of source indices on a shell is the intersection of the sphere $\abs{p}^2=r$ with a translated cube. The central combinatorial step is to decompose these shell slices according to nearest face and dyadic face height, thereby reducing the counting problem to short coordinate intervals together with the classical two-squares representation function.

The principal outputs are as follows. First, the paper derives an exact formula for the mode-level triad count in the cubic truncation and proves the explicit orbit--triad incidence bound in Proposition~\ref{prop:explicit-incidence}. Second, it records the exact orbit-level enstrophy identity in Proposition~\ref{prop:orbit-enstrophy-identity} together with the algebraic decomposition of the transfer matrix. Third, it proves the deterministic row-sum bound for the raw transfer matrix in Theorem~\ref{thm:raw-row-sum} and supplements the theory with the exact finite-$N$ diagnostics collected in Proposition~\ref{prop:finiteN}.

The paper is organized as follows. Section~\ref{sec:framework} fixes notation and the truncated Fourier setting. Section~\ref{sec:lattice} derives the exact mode-level triad count and basic shell arithmetic. Section~\ref{sec:incidence} develops the face-normalized shell-slice decomposition and proves the explicit incidence bound. Section~\ref{sec:enstrophy} records the orbit-level enstrophy identity and the transfer-matrix decomposition. Section~\ref{sec:sobolev} proves the deterministic Sobolev row-sum bound for the raw matrix $M_N(u)$. Section~\ref{sec:finiteN} contains exact finite-$N$ diagnostics.

\section{Framework and notation}
\label{sec:framework}

\subsection{Fourier variables on the torus}

I work on the periodic cube
\begin{equation*}
\Tthree=[0,2\pi]^3.
\end{equation*}
Let
\begin{equation*}
u(x,t)=\sum_{k\in\Z^3\setminus\{0\}}\hat u_k(t)e^{ik\cdot x}
\end{equation*}
be a mean-zero velocity field. The incompressibility and reality constraints are
\begin{equation*}
k\cdot \hat u_k=0,
\qquad
\hat u_{-k}=\overline{\hat u_k}.
\end{equation*}
For $k\neq0$, let
\begin{equation*}
P(k):=I-\frac{k\otimes k}{\abs{k}^2}
\end{equation*}
denote the Fourier-side Leray projector.

\subsection{Cubic Fourier--Galerkin truncation}

For an integer $N\ge1$, define the truncated nonzero lattice
\begin{equation*}
\Lambda_N:=\{k\in\Z^3\setminus\{0\}:\abs{k}_\infty\le N\}.
\end{equation*}
The Galerkin-truncated Navier--Stokes system is
\begin{equation*}
\partial_t\hat u_k
=
-\nu \abs{k}^2\hat u_k
-i\sum_{\substack{p\in\Lambda_N\\ q=k-p\in\Lambda_N}}
P(k)\bigl[q(\hat u_p\cdot \hat u_q)\bigr],
\qquad
k\in\Lambda_N.
\end{equation*}

\subsection{Shells, symmetry, and orbit notation}

For each integer $r\ge1$, define the shell
\begin{equation*}
S_r:=\{k\in\Lambda_N:\abs{k}^2=r\},
\end{equation*}
and the set of represented shell radii
\begin{equation*}
\mathcal R_N:=\{\abs{k}^2:k\in\Lambda_N\}.
\end{equation*}

The full octahedral group $\Oh$ acts on $\Z^3$ by signed coordinate permutations. This action preserves $\abs{k}_\infty$ and $\abs{k}$, and therefore acts on $\Lambda_N$ and on each shell $S_r$. The set of $\Oh$-orbits in $\Lambda_N$ is denoted by
\begin{equation*}
\mathcal O_N:=\Lambda_N/\Oh.
\end{equation*}
If $\alpha\in\mathcal O_N$, its cardinality is written $\abs{\alpha}$. By orbit--stabilizer,
\begin{equation*}
1\le \abs{\alpha}\le 48.
\end{equation*}

\subsection{Triad counts}

For $k\in\Lambda_N$, let
\begin{equation*}
T(k,N):=\#\{(p,q)\in(\Lambda_N)^2:p+q=k\}
\end{equation*}
denote the number of admissible ordered mode pairs. For orbit indices $\alpha,\beta\in\mathcal O_N$, define the orbit-pair triad set
\begin{equation*}
\mathcal T_{\alpha\beta}
:=
\{(k,p,q):k\in\alpha,\ p\in\beta,\ q=k-p\in\Lambda_N\},
\end{equation*}
and its cardinality
\begin{equation*}
\Gamma_{\alpha\beta}:=\abs{\mathcal T_{\alpha\beta}}.
\end{equation*}

\section{Lattice geometry and exact triad arithmetic}
\label{sec:lattice}

The first structural input is an exact formula for the mode-level triad count in the cubic truncation.

\begin{theorem}[Exact mode-level triad count]
\label{thm:triad-count}
For every $k=(k_1,k_2,k_3)\in\Lambda_N$,
\begin{equation*}
T(k,N)=\prod_{i=1}^3(2N+1-\abs{k_i})-2.
\end{equation*}
\end{theorem}

\begin{proof}
Fix $k\in\Lambda_N$. A pair $(p,q)$ contributes to $T(k,N)$ if and only if
\begin{equation*}
p,q\in\Lambda_N,
\qquad
p+q=k.
\end{equation*}
Since $q=k-p$, it suffices to count the admissible values of $p$.

For each coordinate $i\in\{1,2,3\}$, the conditions
\begin{equation*}
-N\le p_i\le N,
\qquad
-N\le k_i-p_i\le N
\end{equation*}
imply that $p_i$ must lie in the intersection
\begin{equation*}
[-N,N]\cap[k_i-N,k_i+N].
\end{equation*}
The number of integers in that interval is exactly
\begin{equation*}
2N+1-\abs{k_i}.
\end{equation*}
Because the coordinates decouple, the number of pairs $(p,q)$ with
\begin{equation*}
p+q=k,
\qquad
p,q\in[-N,N]^3
\end{equation*}
is
\begin{equation*}
\prod_{i=1}^3(2N+1-\abs{k_i}).
\end{equation*}
Among these pairs, precisely two use the forbidden zero mode, namely
\begin{equation*}
(p,q)=(0,k)
\qquad\text{and}\qquad
(p,q)=(k,0).
\end{equation*}
Since $0\notin\Lambda_N$, they must be removed. Therefore
\begin{equation*}
T(k,N)=\prod_{i=1}^3(2N+1-\abs{k_i})-2.
\end{equation*}
\end{proof}

\begin{corollary}[Maximum triad count]
\label{cor:max-triad}
There exists $C>0$ such that
\begin{equation*}
\max_{k\in\Lambda_N}T(k,N)\le C N^3.
\end{equation*}
\begin{equation*}
\max_{k\in\Lambda_N}T(k,N)=2N(2N+1)^2-2.
\end{equation*}
\end{corollary}

\begin{proof}
The product
\begin{equation*}
\prod_{i=1}^3(2N+1-\abs{k_i})
\end{equation*}
is maximized, subject to $k\neq0$, when exactly one coordinate has absolute value $1$ and the other two vanish. Thus the maximizing modes are the six axial vectors
\begin{equation*}
(\pm1,0,0),\quad (0,\pm1,0),\quad (0,0,\pm1).
\end{equation*}
For $k=(1,0,0)$,
\begin{equation*}
T(k,N)=2N(2N+1)^2-2,
\end{equation*}
and the $O(N^3)$ estimate follows immediately.
\end{proof}

\begin{lemma}[Number of represented shells]
\label{lem:number-shells}
There exists $C>0$ such that
\begin{equation*}
\abs{\mathcal R_N}\le C N^2.
\end{equation*}
\end{lemma}

\begin{proof}
If $k\in\Lambda_N$, then $\abs{k_i}\le N$ for each $i$, hence
\begin{equation*}
1\le \abs{k}^2\le 3N^2.
\end{equation*}
Every represented shell radius is therefore an integer in $\{1,\dots,3N^2\}$, and so
\begin{equation*}
\abs{\mathcal R_N}\le 3N^2.
\end{equation*}
\end{proof}

\section{Face-normalized shell slices and explicit incidence bounds}
\label{sec:incidence}

\subsection{Shell slices and orbit counts on shells}

Fix a target orbit $\alpha\in\mathcal O_N$, and choose a representative $k\in\alpha$. For each represented shell radius $r\in\mathcal R_N$, define the shell slice
\begin{equation*}
A_r(k):=\{p\in\Lambda_N:\abs{p}^2=r,\ k-p\in\Lambda_N\}.
\end{equation*}
\begin{equation*}
A_r(k)=S_r\cap B(k),
\end{equation*}
where
\begin{equation*}
B(k):=\prod_{j=1}^3[k_j-N,k_j+N]\cap\Z^3.
\end{equation*}
If $\beta\subset S_r$ is a source orbit contained in shell $S_r$, define
\begin{equation*}
m_r(k,\beta):=\abs{A_r(k)\cap\beta}.
\end{equation*}

\begin{lemma}[Equivariance identity]
\label{lem:equivariance}
Let $\alpha,\beta\in\mathcal O_N$, let $k\in\alpha$, and let $r=\abs{p}^2$ for $p\in\beta$. Then
\begin{equation*}
\Gamma_{\alpha\beta}=\abs{\alpha}\,m_r(k,\beta).
\end{equation*}
\begin{equation*}
\sqrt{\Gamma_{\alpha\beta}}\le \sqrt{48}\,\sqrt{m_r(k,\beta)}.
\end{equation*}
\end{lemma}

\begin{proof}
By definition,
\begin{equation*}
\Gamma_{\alpha\beta}
=
\#\{(k',p,q):k'\in\alpha,\ p\in\beta,\ q=k'-p\in\Lambda_N\}.
\end{equation*}
For fixed $k'\in\alpha$, the number of admissible $p\in\beta$ is $\abs{A_r(k')\cap\beta}$. Hence
\begin{equation*}
\Gamma_{\alpha\beta}
=
\sum_{k'\in\alpha}\abs{A_r(k')\cap\beta}.
\end{equation*}
If $k',k''\in\alpha$, then $k''=gk'$ for some $g\in\Oh$. Since $g$ preserves $\Lambda_N$, preserves Euclidean norm, and maps $\beta$ to itself, it induces a bijection
\begin{equation*}
A_r(k')\cap\beta \longleftrightarrow A_r(k'')\cap\beta.
\end{equation*}
Hence the quantity $\abs{A_r(k')\cap\beta}$ is independent of the representative $k'\in\alpha$, and
\begin{equation*}
\Gamma_{\alpha\beta}=\abs{\alpha}\,m_r(k,\beta).
\end{equation*}
The square-root estimate follows from $\abs{\alpha}\le 48$.
\end{proof}

\subsection{Face heights and normalized patches}

The cube $B(k)$ has six coordinate faces, indexed by pairs $f=(j,\sigma)$ with
\begin{equation*}
j\in\{1,2,3\},
\qquad
\sigma\in\{\pm1\}.
\end{equation*}
The corresponding face plane is
\begin{equation*}
p_j=k_j+\sigma N.
\end{equation*}
For $p\in B(k)$, define the inward face-height
\begin{equation*}
h_f(p,k):=
\begin{cases}
k_j+N-p_j, & f=(j,+1),\\[1mm]
p_j-(k_j-N), & f=(j,-1).
\end{cases}
\end{equation*}
Set
\begin{equation*}
h_{\min}(p,k):=\min_f h_f(p,k),
\end{equation*}
and let $f(p)$ be the lexicographically first face where this minimum is attained.

If $h_{\min}(p,k)=0$, define $H(p):=1$. If $h_{\min}(p,k)\ge1$, let $H(p)$ be the unique dyadic number such that
\begin{equation*}
H(p)\le h_{\min}(p,k)<2H(p).
\end{equation*}

\begin{definition}[Normalized face patch]
For $r\in\mathcal R_N$, a signed face $f$, and a dyadic scale $1\le H\le N$, define
\begin{equation*}
A_{r,f,H}(k):=
\{p\in A_r(k):f(p)=f,\ H(p)=H\}.
\end{equation*}
\end{definition}

\begin{lemma}[Patch decomposition]
\label{lem:patch-decomposition}
For every $k\in\Lambda_N$ and every $r\in\mathcal R_N$,
\begin{equation*}
A_r(k)=\bigsqcup_f\ \bigsqcup_{\substack{H\ \mathrm{dyadic}\\1\le H\le N}}A_{r,f,H}(k).
\end{equation*}
\end{lemma}

\begin{proof}
Each point $p\in A_r(k)$ has a unique minimizing face $f(p)$ and a unique dyadic height scale $H(p)$. This assigns $p$ to exactly one patch.
\end{proof}

\begin{lemma}[Patch square-root estimate]
\label{lem:patch-sqrt}
For every fixed $k,r,f,H$,
\begin{equation*}
\sum_{\beta\subset S_r}\sqrt{\abs{A_{r,f,H}(k)\cap\beta}}
\le
\abs{A_{r,f,H}(k)}.
\end{equation*}
\end{lemma}

\begin{proof}
For each shell orbit $\beta\subset S_r$, set
\begin{equation*}
m_{r,f,H}(k,\beta):=\abs{A_{r,f,H}(k)\cap\beta}.
\end{equation*}
Since $m_{r,f,H}(k,\beta)$ is a nonnegative integer, $\sqrt{m_{r,f,H}(k,\beta)}\le m_{r,f,H}(k,\beta)$. Summing over $\beta\subset S_r$ yields
\begin{equation*}
\sum_{\beta\subset S_r}\sqrt{m_{r,f,H}(k,\beta)}
\le
\sum_{\beta\subset S_r}m_{r,f,H}(k,\beta)
=
\abs{A_{r,f,H}(k)}.
\end{equation*}
\end{proof}

\begin{lemma}[Distinguished coordinate interval]
\label{lem:coordinate-interval}
Fix $k\in\Lambda_N$, a signed face $f=(j,\sigma)$, and a dyadic scale $1\le H\le N$. Then the distinguished coordinate $\nu:=p_j$ of every point $p\in A_{r,f,H}(k)$ belongs to an interval containing at most $C H$ integers, where $C$ is an absolute constant.
\end{lemma}

\begin{proof}
Assume first that $f=(j,+1)$. Then $h_f(p,k)=k_j+N-p_j$.

If $H=1$, the condition $H(p)=1$ implies $0\le h_f(p,k)<2$, so
\begin{equation*}
k_j+N-1\le p_j\le k_j+N.
\end{equation*}
This interval contains at most two integers.

If $H>1$, the condition $p\in A_{r,f,H}(k)$ implies
\begin{equation*}
H\le h_f(p,k)<2H,
\end{equation*}
hence
\begin{equation*}
H\le k_j+N-p_j<2H.
\end{equation*}
\begin{equation*}
k_j+N-2H<p_j\le k_j+N-H.
\end{equation*}
This interval contains at most $2H$ integers.

The case $f=(j,-1)$ is analogous. There one has $h_f(p,k)=p_j-(k_j-N)$, and the same argument gives an interval of length $O(H)$. This proves the claim.
\end{proof}

\begin{lemma}[Patchwise two-squares bound]
\label{lem:two-squares-patch}
For every $\varepsilon>0$ there exists $C_\varepsilon>0$ such that
\begin{equation*}
\abs{A_{r,f,H}(k)}\le C_\varepsilon H N^\varepsilon
\end{equation*}
uniformly in $k\in\Lambda_N$, $r\in\mathcal R_N$, signed faces $f$, and dyadic scales $1\le H\le N$.
\end{lemma}

\begin{proof}
Fix $k,r,f,H$, and write $f=(j,\sigma)$. By Lemma~\ref{lem:coordinate-interval}, the distinguished coordinate $u:=p_j$ takes at most $C H$ integer values.

Fix one such value of $u$. Then the shell condition $\abs{p}^2=r$ becomes
\begin{equation*}
p_\ell^2+p_m^2=r-u^2,
\end{equation*}
where $\{\ell,m\}=\{1,2,3\}\setminus\{j\}$. Hence the number of admissible transverse pairs $(p_\ell,p_m)\in\Z^2$ is bounded by the two-squares representation function
\begin{equation*}
r_2(r-u^2):=\#\{(a,b)\in\Z^2:a^2+b^2=r-u^2\}.
\end{equation*}
For every $\varepsilon>0$, the classical divisor-function bound for sums of two squares gives
\begin{equation*}
r_2(n)\le C_\varepsilon n^\varepsilon.
\end{equation*}
see, for example, Hardy and Wright~\cite{HardyWright}.
Since $0\le r\le 3N^2$, I obtain after enlarging $C_\varepsilon$ if necessary
\begin{equation*}
r_2(r-u^2)\le C_\varepsilon N^\varepsilon.
\end{equation*}
Summing over the $O(H)$ admissible values of $u$ yields
\begin{equation*}
\abs{A_{r,f,H}(k)}\le C_\varepsilon H N^\varepsilon.
\end{equation*}
\end{proof}

\subsection{Explicit shell-counting at fixed face and scale}

\begin{lemma}[Explicit shell-count bound for a fixed face-scale patch]
\label{lem:explicit-face-scale-shell-count}
There exists an absolute constant $C>0$ such that for every $k\in\Lambda_N$, every signed face $f$, and every dyadic scale $1\le H\le N$,
\begin{equation*}
\#\{r\in\mathcal R_N:A_{r,f,H}(k)\neq\varnothing\}\le C N^2 H.
\end{equation*}
\end{lemma}

\begin{proof}
Fix $k\in\Lambda_N$, a signed face $f=(j,\sigma)$, and a dyadic scale $H$. Define
\begin{equation*}
\mathcal R_{f,H}(k):=\{r\in\mathcal R_N:A_{r,f,H}(k)\neq\varnothing\}.
\end{equation*}
For each $r\in\mathcal R_{f,H}(k)$, choose one point $p=(p_1,p_2,p_3)\in A_{r,f,H}(k)$. Then
\begin{equation*}
r=p_1^2+p_2^2+p_3^2.
\end{equation*}
By Lemma~\ref{lem:coordinate-interval}, the distinguished coordinate $u:=p_j$ belongs to a set of at most $C_1H$ integers.

Fix one admissible value of $u$. Then
\begin{equation*}
r=u^2+p_\ell^2+p_m^2,
\end{equation*}
where $\{\ell,m\}=\{1,2,3\}\setminus\{j\}$. Since $p\in\Lambda_N$,
\begin{equation*}
-N\le p_\ell,p_m\le N,
\end{equation*}
and therefore
\begin{equation*}
0\le p_\ell^2+p_m^2\le 2N^2.
\end{equation*}
Thus, for fixed $u$, the quantity $p_\ell^2+p_m^2$ can assume at most $2N^2+1$ distinct values. Each such value determines $r$ uniquely, so the number of possible shell radii corresponding to that $u$ is at most $2N^2+1$. Summing over all admissible $u$, I obtain
\begin{equation*}
\abs{\mathcal R_{f,H}(k)}
\le
C_1H(2N^2+1)
\le C N^2 H.
\end{equation*}
\end{proof}

\begin{remark}
The preceding lemma follows from the face-height decomposition alone. A sharper $O(NH)$ shell-window bound would require an additional geometric or arithmetic input.
\end{remark}

\subsection{Explicit incidence bound}

\begin{proposition}[Explicit orbit--triad incidence bound]
\label{prop:explicit-incidence}
For every $\varepsilon>0$ there exists $C_\varepsilon>0$ such that
\begin{equation*}
\max_{\alpha\in\mathcal O_N}
\sum_{\beta\in\mathcal O_N}\sqrt{\Gamma_{\alpha\beta}}
\le
C_\varepsilon N^{4+\varepsilon}.
\end{equation*}
\end{proposition}

\begin{proof}
Fix a target orbit $\alpha\in\mathcal O_N$, and choose a representative $k\in\alpha$. By Lemma~\ref{lem:equivariance},
\begin{equation*}
\sum_{\beta\in\mathcal O_N}\sqrt{\Gamma_{\alpha\beta}}
\le
\sqrt{48}
\sum_{r\in\mathcal R_N}\sum_{\beta\subset S_r}\sqrt{m_r(k,\beta)}.
\end{equation*}
Next decompose $A_r(k)$ into normalized face patches:
\begin{equation*}
A_r(k)=\bigsqcup_f\ \bigsqcup_H A_{r,f,H}(k).
\end{equation*}
For each shell orbit $\beta\subset S_r$,
\begin{equation*}
m_r(k,\beta)=\sum_{f,H}\abs{A_{r,f,H}(k)\cap\beta}.
\end{equation*}
Using the subadditivity of the square root and then Lemma~\ref{lem:patch-sqrt},
\begin{equation*}
\sum_{\beta\subset S_r}\sqrt{m_r(k,\beta)}
\le
\sum_f\sum_H\sum_{\beta\subset S_r}\sqrt{\abs{A_{r,f,H}(k)\cap\beta}}
\le
\sum_f\sum_H \abs{A_{r,f,H}(k)}.
\end{equation*}
Therefore
\begin{equation*}
\sum_{r\in\mathcal R_N}\sum_{\beta\subset S_r}\sqrt{m_r(k,\beta)}
\le
\sum_f\sum_H
\sum_{\substack{r\in\mathcal R_N\\ A_{r,f,H}(k)\neq\varnothing}}
\abs{A_{r,f,H}(k)}.
\end{equation*}
Fix $f$ and $H$. By Lemma~\ref{lem:explicit-face-scale-shell-count},
\begin{equation*}
\#\{r\in\mathcal R_N:A_{r,f,H}(k)\neq\varnothing\}\le C N^2 H.
\end{equation*}
For each such shell, Lemma~\ref{lem:two-squares-patch} gives
\begin{equation*}
\abs{A_{r,f,H}(k)}\le C_\varepsilon H N^\varepsilon.
\end{equation*}
Hence
\begin{equation*}
\sum_{\substack{r\in\mathcal R_N\\ A_{r,f,H}(k)\neq\varnothing}}
\abs{A_{r,f,H}(k)}
\le
C_\varepsilon N^{2+\varepsilon}H^2.
\end{equation*}
There are only six signed faces, so
\begin{equation*}
\sum_{r\in\mathcal R_N}\sum_{\beta\subset S_r}\sqrt{m_r(k,\beta)}
\le
C_\varepsilon N^{2+\varepsilon}
\sum_{\substack{H\ \mathrm{dyadic}\\1\le H\le N}} H^2.
\end{equation*}
Since the dyadic sum satisfies
\begin{equation*}
\sum_{\substack{H\ \mathrm{dyadic}\\1\le H\le N}} H^2\le C N^2,
\end{equation*}
it follows that
\begin{equation*}
\sum_{\beta\in\mathcal O_N}\sqrt{\Gamma_{\alpha\beta}}
\le
C_\varepsilon N^{4+\varepsilon}.
\end{equation*}
The estimate is uniform in $\alpha$, proving the proposition.
\end{proof}

\section{Orbit-level enstrophy dynamics}
\label{sec:enstrophy}

This section derives the exact orbit-level identity for the truncated enstrophy dynamics and isolates the symmetric part of the transfer matrix.

Define the truncated enstrophy
\begin{equation*}
Z_N(t):=\frac12\sum_{k\in\Lambda_N}\abs{k}^2\abs{\hat u_k(t)}^2.
\end{equation*}
For an orbit $\alpha\in\mathcal O_N$, define the orbit enstrophy and orbit dissipation
\begin{equation*}
Z_\alpha(t):=
\frac{1}{2\abs{\alpha}}
\sum_{k\in\alpha}\abs{k}^2\abs{\hat u_k(t)}^2,
\qquad
D_\alpha(t):=
\frac{1}{\abs{\alpha}}
\sum_{k\in\alpha}\abs{k}^4\abs{\hat u_k(t)}^2.
\end{equation*}
Then
\begin{equation*}
Z_N(t)=\sum_{\alpha\in\mathcal O_N}\abs{\alpha}\,Z_\alpha(t).
\end{equation*}

For orbit pairs $(\alpha,\beta)$, define the raw orbit-level transfer coefficient
\begin{equation*}
M_{\alpha\beta}(u)
:=
\frac{1}{\abs{\alpha}}
\sum_{k\in\alpha}
\sum_{\substack{p\in\beta\\ q=k-p\in\Lambda_N}}
\abs{k}^2
\Re\!\left(
\overline{\hat u_k}\cdot[-iP(k)(q(\hat u_p\cdot\hat u_q))]
\right).
\end{equation*}

\begin{proposition}[Orbit-level enstrophy identity]
\label{prop:orbit-enstrophy-identity}
For every smooth Galerkin solution and every orbit $\alpha\in\mathcal O_N$,
\begin{equation*}
\frac{d}{dt}Z_\alpha(t)
=
-\nu D_\alpha(t)
+
\sum_{\beta\in\mathcal O_N}M_{\alpha\beta}(u).
\end{equation*}
Consequently, summing over $\alpha$ with weights $\abs{\alpha}$ recovers the truncated enstrophy balance
\begin{equation*}
\frac{d}{dt}Z_N(t)
=
-\nu\sum_{k\in\Lambda_N}\abs{k}^4\abs{\hat u_k}^2
+
\sum_{k\in\Lambda_N}\abs{k}^2
\Re\!\left(
\overline{\hat u_k}\cdot
\left[
-i\sum_{\substack{p\in\Lambda_N\\ q=k-p\in\Lambda_N}}
P(k)\bigl(q(\hat u_p\cdot\hat u_q)\bigr)
\right]
\right).
\end{equation*}
\end{proposition}

\begin{proof}
Differentiate the definition of $Z_\alpha(t)$ along the Galerkin system and substitute the truncated Fourier equation for each $k\in\alpha$. The viscous contribution is exactly $-\nu D_\alpha(t)$. Grouping the nonlinear terms according to the source orbit $\beta$ containing the index $p$ yields the stated formula for $dZ_\alpha/dt$. Summing over all $\alpha$ with weights $\abs{\alpha}$ reconstructs the full lattice sum and hence the displayed identity for $Z_N(t)$.
\end{proof}

Define
\begin{equation*}
A_{\alpha\beta}(u):=\frac{M_{\alpha\beta}(u)-M_{\beta\alpha}(u)}{2},
\qquad
V_{\alpha\beta}(u):=\frac{M_{\alpha\beta}(u)+M_{\beta\alpha}(u)}{2}.
\end{equation*}
In matrix form,
\begin{equation*}
M_N(u)=A_N(u)+V_N(u),
\qquad
A_N(u)^T=-A_N(u),
\qquad
V_N(u)^T=V_N(u).
\end{equation*}
Thus $V_N(u)$ is the symmetric part entering the quadratic form on the Euclidean orbit-coordinate space, while $A_N(u)$ is the antisymmetric redistribution part.

\begin{remark}
The decomposition above is purely algebraic: it does not assert that $M_N(u)$ acts as a closed linear operator on the orbit enstrophy vector $(Z_\alpha)_\alpha$. The entries of $M_N(u)$ still depend on the full Fourier state through the triadic products $\hat u_p\cdot\hat u_q$, so the matrix should be viewed as a state-dependent state-dependent algebraic object for the quadratic nonlinearity.
\end{remark}

\begin{remark}
Any uniform row-sum estimate for the symmetric part $V_N(u)$ also requires control of the column sums of $M_N(u)$, since
\begin{equation*}
\sum_\beta \abs{V_{\alpha\beta}(u)}
\le
\frac12\sum_\beta \abs{M_{\alpha\beta}(u)}+
\frac12\sum_\beta \abs{M_{\beta\alpha}(u)}.
\end{equation*}
Accordingly, the deterministic estimate in Section~\ref{sec:sobolev} is stated only for the raw matrix $M_N(u)$.
\end{remark}

\section{Deterministic Sobolev row-sum bounds}
\label{sec:sobolev}

This section proves a deterministic row-sum bound for the raw orbit-level transfer matrix under Sobolev regularity.

\begin{lemma}[Pointwise Sobolev decay]
\label{lem:sobolev-pointwise}
Let $s\in\R$. If $u\in H^s(\Tthree)$ and $\norm{u}_{H^s}\le M$, then for every $k\neq0$,
\begin{equation*}
\abs{\hat u_k}\le M\abs{k}^{-s}.
\end{equation*}
\end{lemma}

\begin{proof}
Since
\begin{equation*}
\norm{u}_{H^s}^2=\sum_{k\neq0}\abs{k}^{2s}\abs{\hat u_k}^2,
\end{equation*}
each summand is bounded by $M^2$. Therefore $\abs{\hat u_k}\le M\abs{k}^{-s}$ for every $k\neq0$.
\end{proof}

\begin{proposition}[Weighted orbit-pair estimate]
\label{prop:weighted-orbit-pair}
Let $s\in\R$, and let $u$ be divergence-free with $\norm{u}_{H^s}\le M$. Then for every pair of orbits $(\alpha,\beta)$,
\begin{equation*}
\abs{M_{\alpha\beta}(u)}
\le
M^3\abs{k_\alpha}^{2-s}
\frac{1}{\abs{\alpha}}
\sum_{(k,p,q)\in\mathcal T_{\alpha\beta}}
\abs{p}^{-s}\abs{q}^{1-s},
\end{equation*}
where $k_\alpha$ denotes any representative of $\alpha$.
\end{proposition}

\begin{proof}
For each triad $(k,p,q)\in\mathcal T_{\alpha\beta}$,
\begin{equation*}
\left|
\abs{k}^2
\Re\!\left(
\overline{\hat u_k}\cdot[-iP(k)(q(\hat u_p\cdot\hat u_q))]
\right)
\right|
\le
\abs{k}^2\abs{q}\abs{\hat u_k}\abs{\hat u_p}\abs{\hat u_q},
\end{equation*}
because $\norm{P(k)}_{\mathrm{op}}\le1$. Applying Lemma~\ref{lem:sobolev-pointwise} to $k$, $p$, and $q$ gives
\begin{equation*}
\abs{k}^2\abs{q}\abs{\hat u_k}\abs{\hat u_p}\abs{\hat u_q}
\le
M^3\abs{k}^{2-s}\abs{p}^{-s}\abs{q}^{1-s}.
\end{equation*}
All modes in $\alpha$ have the same Euclidean norm, so summing over $\mathcal T_{\alpha\beta}$ and dividing by $\abs{\alpha}$ yields the claim.
\end{proof}

\begin{lemma}[Discrete convolution bound]
\label{lem:discrete-convolution}
Let $3/2<s<3$. For $k\in\Lambda_N$, define
\begin{equation*}
\Sigma_N(k):=
\sum_{\substack{p\in\Lambda_N\\ k-p\in\Lambda_N}}
\abs{p}^{-s}\abs{k-p}^{1-s}.
\end{equation*}
Then there exists $C_s>0$, independent of $N$ and $k$, such that
\begin{equation*}
\Sigma_N(k)\le C_s\bigl(1+\abs{k}^{4-2s}\bigr).
\end{equation*}
\end{lemma}

\begin{proof}
Split the summation domain into
\begin{equation*}
D_1:=\{p\in\Lambda_N:\abs{p}\le 2\abs{k}\},
\qquad
D_2:=\{p\in\Lambda_N:\abs{p}>2\abs{k}\},
\end{equation*}
so that $\Sigma_N(k)=\Sigma_1(k)+\Sigma_2(k)$.

\textbf{Far field.}
If $p\in D_2$, then
\begin{equation*}
\abs{k-p}\ge \abs{p}-\abs{k}>\frac12\abs{p},
\end{equation*}
so
\begin{equation*}
\abs{p}^{-s}\abs{k-p}^{1-s}\le C_s\abs{p}^{1-2s}.
\end{equation*}
Because $1-2s<-2$ when $s>3/2$, the lattice sum $\sum_{p\in\Z^3\setminus\{0\}}\abs{p}^{1-2s}$ converges, and therefore
\begin{equation*}
\Sigma_2(k)\le C_s.
\end{equation*}

\textbf{Near field.}
The points $p=0$ and $p=k$ are excluded from the triad sum, so only points with $\abs{p}\ge1$ and $\abs{k-p}\ge1$ remain. For each such $p$, let $Q_p:=p+[-\tfrac12,\tfrac12]^3$. If $x\in Q_p$, then $\abs{x}\sim\abs{p}$ and $\abs{k-x}\sim\abs{k-p}$, with constants independent of $k$ and $p$. Hence
\begin{equation*}
\Sigma_1(k)
\le
C_s\int_{\abs{x}\le 3\abs{k}} \abs{x}^{-s}\abs{k-x}^{1-s}\,dx.
\end{equation*}
After the rescaling $x=\abs{k}y$, this becomes
\begin{equation*}
\Sigma_1(k)
\le
C_s\abs{k}^{4-2s}
\int_{\abs{y}\le 3}\abs{y}^{-s}\abs{e-y}^{1-s}\,dy,
\end{equation*}
where $e=k/\abs{k}$. The last integral is finite because $\abs{y}^{-s}$ is locally integrable near the origin for $s<3$, $\abs{e-y}^{1-s}$ is locally integrable near $e$ for $s> -1$, and away from those two points the integrand is bounded on the compact set $\{\abs{y}\le3\}$. Therefore
\begin{equation*}
\Sigma_1(k)\le C_s\abs{k}^{4-2s}.
\end{equation*}
Combining the near-field and far-field bounds gives the result.
\end{proof}

\begin{remark}
The proof of Lemma~\ref{lem:discrete-convolution} uses an integral comparison near the origin, in the spirit of standard singular-integral estimates, see, for example, Stein~\cite{Stein}. The argument established here requires in the present argument $s<3$. Extending the estimate to all $s>3/2$ would require an additional discrete input.
\end{remark}

\begin{theorem}[Deterministic row-sum bound for the raw transfer matrix]
\label{thm:raw-row-sum}
Let $3/2<s<3$, and let $u$ be divergence-free with $\norm{u}_{H^s}\le M$. Then there exists $C_s>0$ such that for every orbit $\alpha\in\mathcal O_N$,
\begin{equation*}
\sum_{\beta\in\mathcal O_N}\abs{M_{\alpha\beta}(u)}
\le
C_s M^3\Bigl(\abs{k_\alpha}^{2-s}+\abs{k_\alpha}^{6-3s}\Bigr).
\end{equation*}
\begin{equation*}
\sup_{\alpha\in\mathcal O_N}\sum_{\beta\in\mathcal O_N}\abs{M_{\alpha\beta}(u)}
\le
\begin{cases}
C_s M^3, & 2<s<3,\\[1mm]
C_s M^3 N^{6-3s}, & 3/2<s\le 2.
\end{cases}
\end{equation*}
\end{theorem}

\begin{proof}
Fix $\alpha\in\mathcal O_N$ and choose a representative $k_\alpha\in\alpha$. Summing Proposition~\ref{prop:weighted-orbit-pair} over all source orbits $\beta$ gives
\begin{equation*}
\sum_{\beta\in\mathcal O_N}\abs{M_{\alpha\beta}(u)}
\le
M^3\abs{k_\alpha}^{2-s}
\frac{1}{\abs{\alpha}}
\sum_{k\in\alpha}
\sum_{\substack{p\in\Lambda_N\\ q=k-p\in\Lambda_N}}
\abs{p}^{-s}\abs{q}^{1-s}.
\end{equation*}
For each $k\in\alpha$, the inner sum is exactly $\Sigma_N(k)$. Since $\alpha$ is an $\Oh$-orbit and the kernel $\abs{p}^{-s}\abs{k-p}^{1-s}$ is $\Oh$-invariant, $\Sigma_N(k)$ is constant on $\alpha$. Hence
\begin{equation*}
\sum_{\beta\in\mathcal O_N}\abs{M_{\alpha\beta}(u)}
\le
M^3\abs{k_\alpha}^{2-s}\Sigma_N(k_\alpha).
\end{equation*}
Applying Lemma~\ref{lem:discrete-convolution} yields
\begin{equation*}
\sum_{\beta\in\mathcal O_N}\abs{M_{\alpha\beta}(u)}
\le
C_s M^3\abs{k_\alpha}^{2-s}\bigl(1+\abs{k_\alpha}^{4-2s}\bigr),
\end{equation*}
which is the first claimed estimate.

If $2<s<3$, then both exponents $2-s$ and $6-3s$ are negative, and the bound is uniform because $\abs{k_\alpha}\ge1$. If $3/2<s\le2$, then $6-3s\ge0$, and since $\abs{k_\alpha}\le\sqrt{3}N$,
\begin{equation*}
\abs{k_\alpha}^{2-s}+\abs{k_\alpha}^{6-3s}
\le C_s N^{6-3s}.
\end{equation*}
This proves the theorem.
\end{proof}

\begin{remark}
The preceding theorem bounds only the row sums of the raw matrix $M_N(u)$. To control the symmetric part $V_N(u)$ in the same norm, one would also need column-sum estimates for $M_N(u)$, because
\begin{equation*}
\sum_\beta \abs{V_{\alpha\beta}(u)}
\le
\frac12\sum_\beta \abs{M_{\alpha\beta}(u)}+
\frac12\sum_\beta \abs{M_{\beta\alpha}(u)}.
\end{equation*}
No such two-sided estimate is claimed here.
\end{remark}

\section{Exact finite-$N$ diagnostics}
\label{sec:finiteN}

Several structural quantities of the cubic truncation can be computed exactly at finite $N$.

\begin{proposition}[Exact finite-$N$ diagnostics]
\label{prop:finiteN}
The total number of retained nonzero modes is
\begin{equation*}
\abs{\Lambda_N}=(2N+1)^3-1.
\end{equation*}
The total ordered mode-level triad count satisfies the exact identity
\begin{equation*}
\sum_{k\in\Lambda_N}T(k,N)=(3N^2+3N+1)^3-3(2N+1)^3+2.
\end{equation*}
For $N=1,\dots,8$, the values of $\abs{\mathcal O_N}$, $\abs{\mathcal R_N}$, $\max_k T(k,N)$, and $\sum_{k\in\Lambda_N}T(k,N)$ are listed in Table~\ref{tab:finiteN}.
\end{proposition}

\begin{proof}
The identity for $\abs{\Lambda_N}$ follows by counting lattice points in $[-N,N]^3$ and removing the origin. The formula for $\max_k T(k,N)$ is given by Corollary~\ref{cor:max-triad}. For the total ordered triad count, Theorem~\ref{thm:triad-count} gives
\begin{equation*}
\sum_{k\in\Lambda_N}T(k,N)
=
\sum_{k\in[-N,N]^3\setminus\{0\}}\prod_{i=1}^3(2N+1-\abs{k_i})-2\abs{\Lambda_N}.
\end{equation*}
Since the product is coordinate-separable,
\begin{equation*}
\sum_{k\in[-N,N]^3}\prod_{i=1}^3(2N+1-\abs{k_i})
=
\left(\sum_{m=-N}^N (2N+1-\abs{m})\right)^3.
\end{equation*}
Moreover,
\begin{equation*}
\sum_{m=-N}^N (2N+1-\abs{m})
=(2N+1)+2\sum_{j=1}^N(2N+1-j)
=3N^2+3N+1.
\end{equation*}
Subtracting first the contribution of the excluded zero mode, namely $(2N+1)^3$, and then the term $2\abs{\Lambda_N}=2((2N+1)^3-1)$ gives the stated formula. The orbit and shell counts in Table~\ref{tab:finiteN} are obtained by exact enumeration of $\Lambda_N$ up to the full octahedral action and by exact enumeration of represented values of $\abs{k}^2$.
\end{proof}

\begin{table}[ht]
\centering
\caption{Exact finite-$N$ combinatorial diagnostics for the cubic truncation.}
\label{tab:finiteN}
\begin{tabular}{rrrrrr}
\toprule
$N$ & $\abs{\Lambda_N}$ & $\abs{\mathcal O_N}$ & $\abs{\mathcal R_N}$ & $\max_k T(k,N)$ & $\sum_{k\in\Lambda_N}T(k,N)$ \\
\midrule
1 & 26 & 3 & 3 & 16 & 264 \\
2 & 124 & 9 & 9 & 98 & 6486 \\
3 & 342 & 19 & 18 & 292 & 49626 \\
4 & 728 & 34 & 31 & 646 & 224796 \\
5 & 1330 & 55 & 44 & 1208 & 749580 \\
6 & 2196 & 83 & 66 & 2026 & 2041794 \\
7 & 3374 & 119 & 87 & 3148 & 4816686 \\
8 & 4912 & 164 & 115 & 4622 & 10203576 \\
\bottomrule
\end{tabular}
\end{table}

\begin{remark}
For the orbit counts in Table~\ref{tab:finiteN}, one may take as canonical representative of each $\Oh$-orbit the unique triple $(a,b,c)\in\Z_{\ge0}^3$ with $a\ge b\ge c\ge0$, $a\le N$, and $(a,b,c)\neq(0,0,0)$. Thus $\abs{\mathcal O_N}$ is exactly the number of such triples, while $\abs{\mathcal R_N}$ is the number of distinct integers of the form $a^2+b^2+c^2$ with the same constraints.
\end{remark}

\end{document}